\documentclass[runningheads]{llncs}
\usepackage{xypic,graphicx,enumerate,boxedminipage,amsmath}
\usepackage[T1]{fontenc}
\usepackage{tikz}
\usetikzlibrary{arrows,shapes,calc}

\title{Induced Disjoint Paths and Connected Subgraphs for $H$-Free Graphs\thanks{An extended abstract of this paper will appear in the proceedings of WG 2022~\cite{MPSV22}.}}
\author{Barnaby Martin\inst{2}\orcidID{0000-0002-4642-8614}\and Dani\"el Paulusma\inst{2}\orcidID{0000-0001-5945-9287}  \and Siani Smith\inst{2}\orcidID{0000-0003-0797-0512} \and Erik Jan van Leeuwen\inst{3}\orcidID{0000-0001-5240-7257}}
\authorrunning{B. Martin, D. Paulusma, S. Smith, E.J. van Leeuwen}

\institute{Department of Computer Science,
 Durham University, Durham, UK,
\email{\{barnaby.d.martin,daniel.paulusma,siani.smith\}@durham.ac.uk}
\and
Department of Information and Computing Sciences, 
Utrecht University,\\
The Netherlands,
\email{e.j.vanleeuwen@uu.nl}
}

\newcommand{\ssi}{\subseteq_i}

\newcommand{\NP}{{\sf NP}}
\renewcommand{\P}{{\sf P}}
\newcommand{\FPT}{{\sf FPT}}
\newcommand{\XP}{{\sf XP}}
\newcommand{\Blob}[1]{#1^\circ}

\newcommand{\problemdef}[3]{
	\begin{center}
		\begin{boxedminipage}{1.02\textwidth}
			\textsc{{#1}}\\[1pt]  
			\begin{tabular}{ r p{0.8\textwidth}}
				\textit{~~~~Instance:} & {#2}\\
				\textit{Question:} & {#3}
			\end{tabular}
		\end{boxedminipage}
	\end{center}
}

\begin{document}
\maketitle

\begin{abstract}
Paths $P^1,\ldots,P^k$ in a graph $G=(V,E)$ are mutually induced if any two distinct $P^i$ and $P^j$ have neither common vertices nor adjacent vertices. The {\sc Induced Disjoint Paths} problem is to decide if a graph~$G$ with $k$ pairs of specified vertices $(s_i,t_i)$ contains $k$ mutually induced paths~$P^i$ such that each $P^i$ starts from $s_i$ and ends at $t_i$. This is a classical graph problem that is \NP-complete even for $k=2$. We introduce a natural generalization, {\sc Induced Disjoint Connected Subgraphs}: instead of connecting pairs of terminals, we must connect sets of terminals. We give almost-complete dichotomies of the computational complexity of both problems for $H$-free graphs, that is, graphs that do not contain some fixed graph $H$ as an induced subgraph.  Finally, we give a complete classification of the complexity of the second problem if the number~$k$ of terminal sets is fixed, that is, not part of the input.
\end{abstract}

\section{Introduction}\label{s-intro}

The well-known {\sc Disjoint Paths} problem is one of the problems in Karp's list of \NP-complete problems. It is to decide if a graph has pairwise vertex-disjoint paths $P^1,\ldots,P^k$ where each $P^i$ connects two pre-specified vertices $s_i$ and $t_i$. Its generalization, {\sc Disjoint Connected Subgraphs}, plays a crucial role in the graph minor theory of Robertson and Seymour. This problem asks for connected subgraphs $D^1,\ldots,D^k$, where each $D^i$ connects a pre-specified set of vertices $Z_i$. In a recent paper~\cite{KMPSV22} we classified, subject to a small number of open cases, the complexity of both these problems for {\it $H$-free} graphs, that is, for graphs that do not contain some fixed graph $H$ as an {\it induced} subgraph.

\medskip
\noindent
{\bf Our Focus.} We consider the {\it induced} variants of {\sc Disjoint Paths} and {\sc Disjoint Connected Subgraphs}. These problems behave  differently. Namely, {\sc  Disjoint Paths} for fixed $k$, or more generally, {\sc Disjoint Connected Subgraphs}, after fixing both $k$ and $\ell=\max\{|Z_1|,\ldots,|Z_k|\}$, is polynomial-time solvable~\cite{RS95}. In contrast, {\sc Induced Disjoint Paths} is \NP-complete even when $k=2$, as shown both by Bienstock~\cite{Bi91} and Fellows~\cite{Fe89}. Just as for the classical problems~\cite{KMPSV22}, we perform a systematic study and focus on $H$-free graphs.
As it turns out, for the restriction to $H$-free graphs, the induced variants actually become computationally easier for an infinite family of graphs $H$. We first give some definitions.

\medskip
\noindent
{\bf Terminology.}
For a subset $S\subseteq V$ in a graph $G=(V,E)$, let $G[S]$ denote the {\it induced} subgraph of $G$ by $S$,
that is, $G[S]$ is the graph obtained from $G$ after removing every vertex not in $S$.
Let $G_1+G_2$ be the disjoint union of~two vertex-disjoint graphs $G_1$ and $G_2$.
We say that paths $P^1,\ldots, P^k$, for some $k\geq 1$, 
are {\it mutually induced paths} of $G$ if there exists a set $S\subseteq V$ such that $G[S]=P^1+\ldots +P^k$; note that every $P^i$ is an induced path 
and that there is no edge between two vertices from different paths $P^i$ and $P^j$.
A path $P$ is an {\it $s$-$t$-path} (or $t$-$s$-path) if the end-vertices of $P$ are $s$ and $t$.

A {\it terminal pair} $(s,t)$ is an unordered pair of two distinct vertices~$s$ and $t$ in a graph~$G$, which we call {\it terminals}.
A set $T=\{(s_{1},t_{1}),\ldots,(s_{k},t_{k})\}$ of terminal pairs of~$G$ is a {\it terminal pair collection} if the terminals pairs are pairwise disjoint, 
so, apart from $s_i\neq t_i$ for $i\in \{1,\ldots,k\}$, we also have $\{s_i,t_i\}\cap \{s_j,t_j\}=\emptyset$ for every $1\leq i<j\leq k$.
We now define the following decision problem:
\problemdef{Induced Disjoint Paths}{a graph $G$ and terminal pair collection $T=\{(s_1,t_1)\ldots,(s_k,t_k)\}$.}{does $G$ have a set of mutually induced paths $P^1$,\ldots,$P^k$ such that $P^i$ is an $s_i$-$t_i$ path for $i\in \{1,\ldots,k\}$?}

\noindent
Note that as every path between two vertices $s$ and $t$ contains an induced path between $s$ and $t$, the condition that every $P^i$ must be induced is not strictly needed in the above problem definition.
We say that the paths $P^1,\ldots,P^k$, if they exist, form a {\it solution} of {\sc Induced Disjoint Paths}. 

We now generalize the above notions from pairs and paths to sets and connected subgraphs. 
Subgraphs $D^1,\ldots,D^k$ of a graph $G=(V,E)$ are {\it mutually induced subgraphs } of $G$ if there exists a set $S\subseteq V$ such that $G[S]=D^1+\ldots +D^k$.
A connected subgraph~$D$ of $G$ is a {\it $Z$-subgraph} if $Z\subseteq V(D)$.
A {\it terminal set} $Z$ is an unordered set of distinct vertices, which we again call {\it terminals}. A 
set ${\cal Z}=\{Z_1,\ldots,Z_k\}$ is a {\it terminal set collection} if $Z_1,\ldots,Z_k$ are pairwise disjoint terminal sets. 
We now introduce the generalization:
\problemdef{Induced Disjoint Connected Subgraphs}{a graph $G$ and terminal set collection ${\cal Z}=\{Z_1,\ldots,Z_k\}$.}{does $G$ have a set of mutually induced connected subgraphs $D^1,\ldots, D^k$ such that
$D^i$ is a $Z_i$-subgraph for $i\in \{1,\ldots,k\}$?}

\noindent
The subgraphs $D^1,\ldots,D^k$, if they exist, form a {\it solution}. 
We write  {\sc Induced Disjoint Connected $\ell$-Subgraphs} if $\ell=\max\{|Z_1|,\ldots,|Z_k|\}$ is fixed.
Note that {\sc Induced Disjoint Connected $2$-Subgraphs} is exactly {\sc Induced Disjoint Paths}.

\subsection{Known Results}\label{s-known}

Only results for {\sc Induced Disjoint Paths} are known and these hold for a slightly more general problem definition (see Section~\ref{s-con}). Namely, {\sc Induced Disjoint Paths} is linear-time solvable for circular-arc graphs~\cite{GPV16};
polynomial-time solvable for  chordal graphs~\cite{BGHHKP14}, AT-free graphs~\cite{GPV22}, graph classes of bounded mim-width~\cite{JKT20}; and \NP-complete for claw-free graphs~\cite{FKLP12},  line graphs of triangle-free chordless graphs~\cite{RTV21} and thus for (theta,wheel)-free graphs, and for planar graphs; the last result follows from a result of Lynch~\cite{Ly75} (see~\cite{GPV22}). 
Moreover, {\sc Induced Disjoint Paths} is \XP\ with parameter $k$ for (theta,wheel)-free graphs~\cite{RTV21} and even
 \FPT\ with parameter~$k$ for claw-free graphs~\cite{GPV15} and planar graphs~\cite{KK12}; the latter can be extended to graph classes of bounded genus~\cite{KK09}.

\subsection{Our Results}\label{s-our}	

Let $P_r$ be the path on $r$ vertices. A {\it linear forest} is the disjoint union of one or more paths. We write $F\ssi G$ if $F$ is an induced subgraph of $G$ and $sG$ for the disjoint union of $s$ copies of $G$.
We can now present our first two results:
 the first one includes our dichotomy for {\sc Induced Disjoint Paths} (take $\ell=2$).

\begin{theorem}\label{thm:IDPnew}
Let $\ell\geq 2$. For a graph $H$, {\sc Induced Disjoint Connected $\ell$-Subgraphs} on $H$-free graphs is polynomial-time solvable if  
$H \ssi sP_3+P_6$ for some $s\geq 0$; \NP-complete if $H$ is not a linear forest; and quasipolynomial-time solvable otherwise.
\end{theorem}

\begin{theorem}\label{thm:IDCS}
For a graph $H$ such that $H\neq sP_1+P_6$ for some $s\geq 0$, {\sc Induced Disjoint Connected Subgraphs} on $H$-free graphs
is polynomial-time solvable for $H$-free graphs if $H \ssi sP_1+P_3+P_4$ or $H \ssi sP_1+ P_5$ for some $s\geq 0$, and it is \NP-complete otherwise.
\end{theorem}

\noindent
Note the complexity jumps if we no longer fix $\ell$.  
We will show that all open cases in Theorem~\ref{thm:IDCS} are equivalent to exactly {\bf one} open case, namely $H=P_6$. 

\medskip
\noindent
{\bf Comparison.}
The {\sc Disjoint Connected Subgraphs} problem restricted to $H$-free graphs is polynomial-time solvable if $H\ssi P_4$ and else it is \NP-complete, 
even if the maximum size of the terminal sets is $\ell=2$, except for the three unknown cases $H\in \{3P_1,2P_1+P_2,P_1+P_3\}$~\cite{KMPSV22}. Perhaps somewhat surprisingly, Theorems~\ref{thm:IDPnew} and~\ref{thm:IDCS} show the induced variant is computationally easier for an infinite number of linear forests~$H$ (if $\P\neq \NP$).

\medskip
\noindent
{\bf Fixing \boldmath$k$.} If the number $k$ of terminal sets is fixed, we write {\sc $k$-Induced Disjoint Connected Subgraphs} and
 prove the following complete dichotomy.

\begin{theorem}\label{thm:k-IDCS}
Let $k\geq 2$.
For a graph $H$, {\sc $k$-Induced Disjoint Connected Subgraphs} on $H$-free graphs is polynomial-time solvable for $H$-free graphs if $H \ssi sP_1+2P_4$ or $H \ssi sP_1+ P_6$ for some $s\geq 0$, and it is \NP-complete otherwise.
\end{theorem}

\noindent
{\bf Comparison.}
We note a complexity jump between Theorems~\ref{thm:IDCS} and~\ref{thm:k-IDCS} when $H=sP_1+2P_4$ for some $s\geq 0$. 

\medskip
\noindent
{\bf Paper Outline.} 
Section~\ref{s-pre} contains terminology, known results and auxiliary results that we will use as lemmas.
Hardness results for Theorem~\ref{thm:IDPnew} transfer to Theorem~\ref{thm:IDCS}, whereas the reverse holds for polynomial results. As such, we show all our polynomial-time algorithms in Section~\ref{s-poly} and all our 
hardness reductions in Section~\ref{s-np}. The cases $H=sP_3+P_6$ in Theorem~\ref{thm:IDPnew} and $H=sP_1+P_5$ in Theorem~\ref{thm:IDCS} are proven by a reduction to {\sc Independent Set} via so-called {\it blob graphs}, just as the quasipolynomial-time result if $H$ is a linear forest. Hence, we also include the proof of the latter result in Section~\ref{s-poly}. 
In Section~\ref{s-proofs} we combine the results from the previous two sections to prove Theorems~\ref{thm:IDPnew}--\ref{thm:k-IDCS}.

In our theorems we have infinite families of polynomial cases related to nearly $H$-free graphs.
For a graph~$H$, a graph $G$ is  {\it nearly $H$-free} if $G$ is $(P_1+H)$-free. 
It is easy to see (cf~\cite{BH07}) that {\sc Independent Set} is polynomial-time solvable on nearly $H$-free graphs if it is so on $H$-free graphs. However, for many other graph problems, this might either not be true or less easy to prove (see, 
for example,~\cite{JPP20}). In Section~\ref{s-poly} we show that it holds for the relevant cases in  Theorem~\ref{thm:IDCS}, in particular
 for the case  $H=P_6$ (see Lemma~\ref{l-sp1b}). The latter result yields no algorithm but shows that essentially $H=P_6$ is the only one open case left in 
Theorem~\ref{thm:IDCS} (see also Remark~1). 
 
In Section~\ref{s-pro} we briefly discuss the more general problem definition of {\sc Induced Disjoint Paths} used in the literature. We show that the complexity dichotomy for this problem differs from the one in Theorem~\ref{thm:IDPnew} (for $\ell=2$).

In Section~\ref{s-con} we consider a number of directions for future work. In particular we consider the restriction {\sc $k$-Disjoint Connected $\ell$-Subgraphs} where {\it both} $k$ and $\ell$ are fixed and discuss some open problems. 

\section{Preliminaries}\label{s-pre}

Let $G=(V,E)$ be a graph. A subset $S\subseteq V$ is {\it connected} if $G[S]$ is connected.
 A subset $D\subseteq V(G)$ is {\it dominating} if every vertex of $V(G)\setminus D$ is adjacent to least one vertex of $D$; if $D=\{v\}$ then $v$ is a {\it dominating} vertex.
The {\it open} and {\it closed neighbourhood} of a vertex $u\in V$ are $N(u)=\{v\; |\; uv\in E\}$ and $N[u]=N(u)\cup \{u\}$.
For a set $U\subseteq V$ we define $N(U)=\bigcup_{u\in U}N(u)\setminus U$ and $N[U]=N(U)\cup U$.

For a graph $G=(V,E)$ and a subset $S\subseteq U$, we write $G-S=G[V\setminus S]$. If $S=\{u\}$ for some $u\in V$, we write $G-u$ instead of $G-\{u\}$.
A vertex $u$ is a {\it cut-vertex} of a connected graph $G$ if $G-u$ is disconnected.

The {\it contraction} of an edge $e=uv$ in a graph $G$ replaces the vertices $u$ and $v$ by a new vertex~$w$ that is adjacent to every vertex previously adjacent to $u$ or $v$; note that the resulting graph $G/e$ is still {\it simple}, that is, $G/e$ contains no multi-edges or self-loops.
The following lemma is easy to see  (see, for example,~\cite{KP20}).

\begin{lemma}\label{l-contract}
For a linear forest $H$, let $G$ be an $H$-free graph. Then $G/e$ is $H$-free for every $e\in E(G)$.
\end{lemma}

In a solution $(D^1,\ldots,D^k)$ for an instance $(G,{\cal Z})$ of {\sc Induced Disjoint Connected Subgraphs}, if $D^i$ is minimal and $X_i$ is a minimum connected dominating set of $D^i$, then $X_i \cup Z_i = D^i$ or, equivalently, $D^i \setminus X_i \subseteq Z_i$. This will be relevant in our proofs, where we
use the following result of Camby and Schaudt, in particular for the case $r=6$ (alternatively, we could use the slightly weaker characterization of $P_6$-free graphs in~\cite{HP10} but the below characterization gives a faster algorithm).
  
\begin{theorem}[\cite{CS16}]\label{t-cs}
Let $r\geq 4$ and $G$ be a connected $P_r$-free graph. Let $X$ be any minimum connected dominating set of~$G$. Then 
$G[X]$ is either $P_{r-2}$-free or isomorphic to $P_{r-2}$. 
\end{theorem}

Let $G=(V,E)$ be a graph. Two sets $X_1,X_2 \subseteq V$ are \emph{adjacent} if $X_1 \cap X_2 \neq \emptyset$ or there exists an edge with one end-vertex in $X_1$ and the other in $X_2$. 
The \emph{blob graph} $\Blob{G}$ of $G$ has vertex set
$\{ X \subseteq V(G) ~|~ X \text{ is connected} \}$ and  edge set $\{ X_1X_2 ~|~ X_1 \text{ and } X_2 \text{ are adjacent}\} $. 
Note that blob graphs may have exponential size, but in our proofs we will only construct parts of blob graphs that have polynomial size.
We need the following known lemma that generalizes a result of Gartland et al.~\cite{GLPPR21} for paths.

\begin{lemma}[\cite{PPR21b}]\label{l-blob}
For every linear forest $H$, a graph $G$ is $H$-free if and only if $\Blob{G}$ is $H$-free.
\end{lemma}

The {\sc Independent Set} problem is to decide if a graph $G$ has an {\it independent set} (set of pairwise non-adjacent vertices) of size at least~$k$ for some given integer~$k$. 

We need the following two known results for {\sc Independent Set}. The first one is due to Grzesik, Klimosov\'a, Pilipczuk and Pilipczuk~\cite{GKPP19}. The second  one is due to Pilipczuk, Pilipczuk and Rz{\k a}\.{z}ewski~\cite{PPR21}, who improved the previous quasipolynomial-time algorithm for {\sc Independent Set} on $P_t$-free graphs, due to Gartland and Lokshtanov~\cite{GL20} (whose algorithm runs in  $n^{O(\log^3n)}$ time).

\begin{theorem}[\cite{GKPP19}]\label{t-ip6}
{\sc Independent Set} is polynomial-time solvable for 
$P_6$-free graphs.
\end{theorem}

\begin{theorem}[\cite{GL20}]\label{t-iquasi}
For every $r\geq 1$, {\sc Independent Set} can be solved in $n^{O(\log^2n)}$ time for $P_r$-free graphs.
\end{theorem}

Two instances of some decision problem are {\it equivalent} if one is a yes-instance if and only if the other one is. We make frequently use of the following observation.

\begin{lemma}\label{l-independent}
From an instance $(G,{\cal Z})$ of {\sc Induced Disjoint Connected Subgraphs} we can in linear time, either find a solution for $(G,{\cal Z})$ or
obtain an equivalent instance $(G',{\cal Z}')$ with $|V(G')|\leq |V(G)|$, such that the following holds:
\begin{enumerate}
\item $|{\cal Z}'|\geq 2$;
\item  every $Z_i'\in {\cal Z}'$ has size at least~$2$; and
\item the union of the sets in ${\cal Z}'$ is an independent set.
\end{enumerate}
Moreover, if $G$ is $H$-free for some linear forest $H$, then $G'$ is also $H$-free.
\end{lemma}

\begin{proof}
Let $(G,{\cal Z})$ be an instance of {\sc Induced Disjoint Connected Subgraphs}, where ${\cal Z}=\{Z_1,\ldots,Z_k\}$ for some integer $k\geq 1$.
If two adjacent vertices will always appear in the same set of every solution $(D^1,\ldots,D^k)$, then we can safely contract the edge between them at the start of any algorithm. This property holds for every pair of adjacent vertices of every $Z_i$. Hence, we contract every edge between two vertices that belong to the same set~$Z_i$. This takes linear time.

Let $(G^*,{\cal Z}^*)$ be the resulting instance. Note that every $Z\in {\cal Z}^*$ is an independent set.
If a set $Z\in {\cal Z}^*$ has size~$1$, say $Z=\{z\}$, then we remove $z$ and all its neighbours from $G^*$ to obtain an equivalent instance. After doing this for all singleton sets in ${\cal Z}^*$, we obtain our desired instance $(G',{\cal Z}')$ in linear time.

If $G$ is $H$-free for some linear forest~$H$, then $G'$ is also $H$-free, as $H$-freeness is preserved by edge contraction (Lemma~\ref{l-contract}) and by vertex deletion (by definition).
If it turns out that two vertices from different sets in ${\cal Z}'$ are adjacent, then $(G,{\cal Z})$ is a no-instance. Else we find that the union of the sets in ${\cal Z}'$ is an independent set. If $|{\cal Z}'|=1$, then the problem is trivial to solve, so we may assume that $|{\cal Z}'|\geq 2$. \qed
\end{proof}

We also use the next lemma frequently.

\begin{lemma}\label{l-useful}
Let $H$ be a linear forest.
If $(G,{\cal Z})$ is a yes-instance of {\sc Induced Disjoint Connected Subgraphs} and $G$ is $H$-free, then
$(G,{\cal Z})$ has a solution $(D^1,\ldots,D^k)$, where each $D^i$ has size at most $(2|V(H)|-1)|Z_i|$.
\end{lemma}

\begin{proof}
Consider a solution $(D^1,\ldots,D^k)$. As $G$ is $H$-free and $H$ is a linear forest, $G$ is also $P_t$-free where $t=2|V(H)|-1$. For every $D^i$, fix a vertex $u$ in it. As $D^i$ is connected and $G$ is $P_t$-free, there exists a path from every vertex of $Z_i$ to~$u$ that has length at most $t-1$. Let $A^i$ be the subgraph of $D^i$ induced by the union of the vertex sets of these paths of $D^i$. Note that the number of vertices of each~$A^i$ is at most $t|Z_i|$. As each $A^i$ is connected, $(A^1,\ldots, A^k)$ is a solution. \qed
\end{proof}

\section{Algorithms}\label{s-poly}

In this section we show all the polynomial-time and quasipolynomial-time results needed to prove our main theorems. 

\subsection{Using the Blob Graph Approach}

We start with the following result, which holds for every fixed integer~$\ell\geq 2$. In the proof of this result and the next one we use the blob graph approach.

\begin{lemma}\label{l-sp6}
Let $\ell\geq 2$. For every $s\geq 0$, {\sc Induced Disjoint Connected $\ell$-Subgraphs} is polynomial-time solvable for 
$(sP_3+P_6)$-free graphs.
\end{lemma}

\begin{proof}
Let $(G,{\cal Z})$ be an instance of 
the {\sc Induced Disjoint Connected $\ell$-Subgraphs} problem, where $G$ is an $(sP_3+P_6)$-free graph for some $s\geq 0$. 
By Lemma~\ref{l-independent}, we may assume that the union of the sets in ${\cal Z}=\{Z_1,\ldots,Z_k\}$ is an independent set. 

First suppose that $k\leq s$. By Lemma~\ref{l-useful} we may assume that each $D^i$ in a solution $(D^1,\ldots,D^k)$ has size at most $t=(6s+11)\ell$. So $|D^1|+\ldots+|D^k|$ has size at most $kt\leq st$. Hence, we can consider all $O(n^{st})$ options of choosing a solution. As $s$ and $t$ are constants, this takes polynomial time in total.

Now suppose that $k\geq s+1$. We consider all $O(n^{(s-1)t})$ options of choosing the first $s$ subgraphs $D^i$, discarding those with an edge between distinct $D^i$ or between some $D^i$ and $Z_j$ for some $j\geq s+1$. For each remaining option, let $G'=G-N[V(D^1)\cup \cdots \cup V(D^s)]$ and ${\cal Z}'=\{Z_{s+1},\ldots,Z_k\}$. Note that  $G'$ is $P_6$-free.

Let $F$ be the subgraph of the blob graph $\Blob{G'}$ induced by all connected subsets $X$ in $G'$ that have size at most $11\ell$, such that $X$ contains all vertices of one set from ${\cal Z}'$ and no vertices from any other set of ${\cal Z}'$. Then $F$ has polynomial size, as it has $O(n^{11\ell})$ vertices, so we can construct $F$ in polynomial time.
By Lemma~\ref{l-blob}, $F$ is $P_6$-free.

We claim that $(G',{\cal Z}')$ has a solution if and only if $F$ has an independent set of size $k-s$.

First suppose that $(G',{\cal Z}')$ has a solution. Then, by Lemma~\ref{l-useful}, it has a solution $(D^{s+1},\ldots,D^k)$, where each $D^i$ has size at most~$11\ell$. Such a solution corresponds to an independent set of 
size~$k-s$ in~$F$. For the reverse implication, two vertices in $F$ that each contain vertices of the same set $Z_i$ are adjacent.
Hence, an independent set of size~$k-s$ in $F$ is a solution for $(G',{\cal Z}')$. 

Due to the above, it remains to apply Theorem~\ref{t-ip6} to find in polynomial time
whether $\Blob{G'}$ has an independent set of size $k-s$.\qed
\end{proof}
By replacing Theorem~\ref{t-ip6} by Theorem~\ref{t-iquasi} in the above proof and repeating the arguments of the second part we obtain the following result.

\begin{lemma}\label{l-quasi}
Let $\ell\geq 2$. For every $r\geq 1$, {\sc Induced Disjoint Connected $\ell$-Subgraphs} is quasipolynomial-time solvable for  $P_r$-free graphs.
\end{lemma}

\subsection{Nearly H-Free Graphs}

In this section we prove a crucial lemma on nearly $H$-free graphs.

\begin{lemma}\label{l-sp1b}
For $k\geq 2$, $r\leq 6$ and $s\geq 1$, if {\sc ($k$-)Induced Disjoint Connected Subgraphs} is polynomial-time solvable for $P_r$-free, graphs, then it is so for $(sP_1+P_r)$-free graphs.
\end{lemma}

\begin{proof}
First let $r=6$ and $k$ be part of the input.
Let $(G,{\cal Z})$ be an instance of  {\sc Induced Disjoint Connected Subgraphs}, where $G$ is an $(sP_1+P_6)$-free graph for some integer $s\geq 1$ and ${\cal Z}=\{Z_1,\ldots,Z_k\}$. We may assume without loss of generality that $|Z_1|\geq |Z_2|\geq \cdots \geq |Z_k|$.
By Lemma~\ref{l-independent}, we may assume that $k\geq 2$; every $Z_i\in {\cal Z}$ has size at least~$2$; and the union of the sets in ${\cal Z}$ is an independent set. We assume that {\sc Induced Disjoint Connected Subgraphs} is polynomial-time solvable for $P_6$-free graphs.

\medskip
\noindent
{\bf Case 1.} For every $i\geq 2$, $|Z_i|\leq s-1$.\\
Let $D^1,\ldots,D^k$ be a solution for $(G,{\cal Z})$ (assuming it exists). By Lemma~\ref{l-useful}, we may assume without loss of generality that for $i\geq 2$, the number of vertices of~$D^i$ is at most $(2s+11)|Z_i|\leq (2s+11)(s-1)$.

First assume $k\leq s$. Then $V(D^2)\cup \cdots \cup V(D^k)$ has size at most $t$, where $t=(s-1)(2s+11)(s-1)$ is a constant. Hence, we can do as follows. We consider all $O(n^t)$ options for choosing the subgraphs $D^2,\ldots,D^k$. For each choice we check in polynomial time if
$D^2,\ldots,D^k$ are mutually induced and connected, and if each $D^i$ contains $Z^i$.
We then check in polynomial time if the graph $G-N[(V(D^2)\cup \cdots V(D^k)]$ has a connected component containing $Z_1$. As the number of choices is polynomial, the total running time is polynomial.

Now assume $k\geq s+1$. We consider all $O(n^{s(2s+11)(s-1)})$ options of choosing the $s$ subgraphs $D^2,\ldots,D^{s+1}$. We discard an option if for some $i\in \{1,\ldots,s\}$, the graph $D^i$ is disconnected. We also discard an option if there is an edge between two vertices from two different subgraphs $D^h$ and $D^i$ for some $2\leq h<i\leq s+1$, or if there is an edge between a vertex from some subgraph $D^h$ ($2\leq h\leq s$) and a vertex from some set $Z_i$ $(i=1$ or $i\geq s+2$). If we did not discard the option, then we solve {\sc Induced Disjoint Connected Subgraphs} on instance $(G-\bigcup_{i=2}^{s+2} N[V(D^i)], {\cal Z}\setminus \{Z_2,\ldots,Z_{s+1}\})$. The latter takes polynomial time as $G-\bigcup_{i=2}^{s+1} N[D^i]$ is $P_6$-free. As the number of branches is polynomial as well, the total running time is polynomial.

\medskip
\noindent
{\bf Case 2.} $|Z_2|\geq s$ (and thus also $|Z_1|\geq s$).\\
Let $D^1,\ldots,D^k$ be a solution for $(G,{\cal Z})$ (assuming it exists). As $|Z_1|\geq s$, we find that for every $i\geq 2$, $D^i$ is $P_6$-free. As $|Z_2|\geq s$, we also find that $D^1$ is $P_6$-free. Then, by setting 
$r=6$ in Theorem~\ref{t-cs}, every $D^i$ ($i\in \{1,\ldots, k\}$) has a connected dominating set~$X_i$ such that $G[X_i]$ is either $P_4$-free or isomorphic to $P_4$. 
We may assume without loss of generality that every $X_i$ is inclusion-wise minimal (as otherwise we could just replace $X_i$ by a smaller connected dominating set of $D^i$). 

\medskip
\noindent
{\bf Case 2a.} There exist some $X_i$ with size at least $7s+2$.\\
As $s\geq 1$, we have that $G[X_i]$ is $P_4$-free. We now set 
$r=4$ in Theorem~\ref{t-cs} and find that $G[X_i]$ has a connected dominating set~$Y_i$ of size at most $2$. Hence, $G[X_i]$ contains a set $R$ of $7s$ vertices that are not cut-vertices of $G[X_i]$.
As $X_i$ is minimal, this means that in $D^i$, each $r\in R$ has at least one neighbour $z\in Z_i$ that is not adjacent to any vertex of $X_i\setminus \{r\}$. We say that $z$ is a {\it private} neighbour of $r$. We now partition $R$ into sets $R_1,\ldots,R_7$, each of exactly $s$ vertices. For $h=1,\ldots,7$, let $R_h=\{r_h^1,\ldots,r_h^s\}$ and pick a private neighbour $z_h^j$ of $r_h^j$. For $h=1,\ldots,7$, let $Q_h=\{z_h^1,\ldots,z_h^s\}$. Each $Q_h$ is independent, as $Z_i$ is independent and $Q_h\subseteq Z_i$.

We claim that there exists an index $h\in \{1,\ldots,7\}$ such that $G-(N[Q_h]\setminus R_h)$ is $P_6$-free. For a contradiction, assume that for every $h\in \{1,\ldots,7\}$, we have that $G-(N[Q_h]\setminus R_h)$ is not $P_6$-free. As $G$ is $(sP_1+P_6)$-free and every $Q_h$ is an independent set of size $s$, we have that $G-N[Q_h]$ is $P_6$-free. We conclude that every induced $P_6$ of $G$ contains a vertex of $R_h$ for every $h\in \{1,\ldots,7\}$. This is contradiction, as every induced $P_6$ only has six vertices. Hence, there exists an index $h\in \{1,\ldots,7\}$ such that $G-(N[Q_h]\setminus R_h)$ is $P_6$-free.

We exploit the above structural claim algorithmically as follows. We consider all $k=O(n)$ options that one of the sets $X_i$ has size at least $7s+2$. For each choice of index $i$ we do as follows. We consider all $O(n^{2s})$ options of choosing a set $Q_h$ of $s$ vertices from the independent set $Z_i$ together with a set $R_h$ of $s$ vertices from $N(Q_h)$. We discard the option if a vertex of $Q_h$ has more than one neighbour in $R_h$, or if  $G'=G-(N[Q_h]\setminus R_h)$ is not $P_6$-free. Otherwise, we solve {\sc Induced Disjoint Connected Subgraphs} on instance $(G',{\cal Z}')$, where ${\cal Z}'=({\cal Z}\setminus \{Z_i\}) \cup 
\{(Z_i\setminus Q_h) \cup R_h\}$. As $G'$ is $P_6$-free, the latter takes polynomial time by our initial assumption. Hence, as the total number of branches is $O(n^{2s+1})$ the total running time of this check takes polynomial time.

\medskip
\noindent
{\bf Case 2b.} Every $X_i$ has size at most $7s+1$.\\ 
First assume $k\leq s$. We consider all $O(n^{s(7s+1)})$ options of choosing the sets $X_1,\ldots,X_k$. For each option we check if $(X_1\cup Z_1,\ldots,X_k\cup Z_k)$ is a solution for
$(G,{\cal Z})$. As the latter takes polynomial time and the total number of branches is polynomial, this takes polynomial time.

Now assume $k\geq s+1$. We consider all $O(n^{s(7s+1)})$ options of choosing the first $s$ sets $X_1,\ldots,X_s$. We discard an option if for some $i\in \{1,\ldots,s\}$, the set $X_i\cup Z_i$ is disconnected. We also discard an option if there is an edge between two vertices from two different sets 
$X_h\cup Z_h$ and $X_i\cup Z_i$ for some $1\leq h<i\leq s$, or if there is an edge between a vertex from some set $X_h\cup Z_h$ ($h\leq s$) and a vertex from some set $Z_i$ $(i\geq s+1$). If we did not discard the option, then we solve {\sc Induced Disjoint Connected Subgraphs} on instance $(G-\bigcup_{i=1}^s N[X_i\cup Z_i], \{Z_{s+1},\ldots, Z_k\})$. The latter takes polynomial time as $G-\bigcup_{i=1}^s N[X_i\cup Z_i]$ is $P_6$-free. As the number of branches is polynomial as well, the total running time is polynomial.

\medskip
\noindent
From the above case analysis we conclude that the running time of our algorithm is polynomial. If $r\leq 5$ and/or $k$ is fixed we use exactly the same arguments.
\qed
\end{proof}

\noindent
{\bf Remark 1.}
Due to Lemma~\ref{l-sp1b}, the missing cases $H=sP_1+P_6$ in Theorem~\ref{thm:IDCS} are all equivalent to the case $H=P_6$. 

\subsection{Two Applications of Lemma~\ref{l-sp1b}}

We will first use Lemma~\ref{l-sp1b} for the case where $r=5$. In the proof of the next result, we also make use of the blob approach again.

\begin{lemma}\label{l-p5}
For every $s\geq 0$, {\sc Induced Disjoint Connected Subgraphs} is polynomial-time solvable for $(sP_1+P_5)$-free graphs.
\end{lemma}

\begin{proof}
Due to Lemma~\ref{l-sp1b} it suffices to prove the statement for $P_5$-free graphs only.
Let $(G,{\cal Z})$ be an instance of  {\sc Induced Disjoint Connected Subgraphs}, where $G$ is a $P_5$-free graph and ${\cal Z}=\{Z_1,\ldots,Z_k\}$. 
By Lemma~\ref{l-independent}, we may assume that $k\geq 2$; every $Z_i\in {\cal Z}$ has size at least~$2$; and the union of the sets in ${\cal Z}$ is an independent set. We may also delete every vertex from $G$ that is not in a terminal set from ${\cal Z}$ but that is adjacent to two terminals in different sets $Z_h$ and $Z_i$ (such a vertex cannot be used in any subgraph of a solution). We now make a structural observation that gives us a procedure for safely contracting edges; recall that edge contraction preserves $P_5$-freeness by Lemma~\ref{l-contract}. 

Consider a solution $(D^1 \dots D^k)$ that is {\it maximal} in the sense that any vertex~$v$ outside $V(D^1)\cup \cdots \cup V(D^k)$ must have a neighbour in
at least two distinct subgraphs~$D^i$ and $D^j$. Since $G$ is $P_5$-free, $v$ must be adjacent to all vertices of at least one of $D^i$ and $D^j$.
Since $v$ does not have neighbours in both $Z_i\subseteq V(D^i)$ and $Z_j\subseteq V(D^j)$, we find that  $v$ is adjacent to all vertices of exactly one of $D^i$ and~$D^j$.

The above gives rise to the following algorithm. Let $v$ be a vertex that is adjacent to at least one vertex $z \in Z_i$ but not to all vertices of $Z_i$.
As $v$ is adjacent to $z$ and $z$ is in $Z_i$, it hold that $v$ does not belong to any $D^h$ with $h\neq i$ for every (not necessarily maximal) solution $(D^1,\ldots,D^k)$. The observation from the previous paragraph tells us that if $v$ is not in any $D^h$ and $(D^1,\ldots,D^k)$ is a maximal solution,
then $v$ must be adjacent to all vertices of some $D^j$. As $v$ is adjacent to $z\in Z_i$, it holds by construction that $v$ is not adjacent to any vertex of any $Z_h\subseteq V(D^h)$ with $h\neq i$. Hence,  $i=j$ must hold. However, this is not possible, as we assumed that $v$ is not adjacent to all vertices of $Z_i\subseteq V(D^i)$.
Hence, we may assume without loss of generality that $v$ belongs to $D^i$ (should a solution exist).
This means that we can safely contract the edge $vz$ and put the resulting vertex in~$Z_i$. Then we apply Lemma~\ref{l-independent} again and also remove all common neighbours of vertices from $Z_i$ and vertices from other sets $Z_j$. This takes polynomial time and the resulting graph has one vertex less. Hence, by applying this procedure exhaustively we have, in polynomial time, either solved the problem or obtained an equivalent but smaller instance.

Suppose we have an equivalent instance. For simplicity we denote the obtained instance by $(G,{\cal Z})$ again, where $G$ is a $P_5$-free graph
 and ${\cal Z}=\{Z_1,\ldots,Z_k\}$ with $k\geq 2$. Due to our procedure, every $Z_i\in {\cal Z}$ has size at least~$2$;  the union of the sets in ${\cal Z}$ is an independent set. Moreover, every non-terminal vertex is adjacent either to no terminal vertex or is adjacent to all terminals of exactly one terminal set. We let $S$ be the set of vertices of the latter type. Observe that it follows from the preceding that only vertices of $S$ need to be used for a solution.

We now construct the subgraph $F$ of the blob graph $\Blob{G}$ that is induced by all connected subsets $X$ of the form $X=Z_i\cup \{s\}$ for some $1\leq i\leq k$ and $s\in S$. Note that $F$ has $O(kn)$ vertices. Hence, constructing $F$ takes polynomial time.
 Moreover, $F$ is $P_5$-free due to Lemma~\ref{l-blob}. As in the proof of Lemma~\ref{l-sp6}, we observe that $(G,{\cal Z})$ has a solution if and only if $F$ has an independent set of size~$k$. It now remains to apply (in polynomial time) Theorem~\ref{t-ip6}. \qed
\end{proof}

\noindent
We now show a stronger result when $k$ is fixed instead of part of the input. Again we will use Lemma~\ref{l-sp1b}.

\begin{lemma}\label{l-kp6}
For every integer $s\geq 0$, {\sc $k$-Induced Disjoint Connected Subgraphs} is polynomial-time solvable for $(sP_1+P_6)$-free graphs.
\end{lemma}

\begin{proof}
Due to Lemma~\ref{l-sp1b} it suffices to prove the statement for $P_6$-free graphs only.
Let $(G,{\cal Z})$ be an instance of {\sc $k$-Induced Disjoint Connected Subgraphs}, where $G$ is a $P_6$-free graph  and ${\cal Z}=\{Z_1,\ldots,Z_k\}$. By Lemma~\ref{l-independent}, we may assume that every $Z_i\in {\cal Z}$ has size at least~$2$ and that the union of the sets in ${\cal Z}$ is an independent set.
We start by deleting from $G$,
\begin{itemize}
\item [(i)] every common neighbour of a vertex of  $Z_i$ and a vertex of $Z_j$ with $i\neq j$. 
\end{itemize}
This takes polynomial time and is safe, as such vertices are not in any subgraph of a solution.

Consider a solution $(D^1,\ldots,D^k)$ (if it exists). As $G$ is $P_6$-free, each $D^i$ is $P_6$-free. By Theorem~\ref{t-cs} (take $r=6$), this means that every $D^i$ has a connected dominating set~$X_i$ such that $G[X_i]$ is either $P_4$-free or isomorphic to $P_4$. 
In the former case we say that $D^i$ is {\it difficult} and in the latter we say that $D^i$ is {\it easy}. 

We consider all options of choosing which of the subgraphs $D^i$ is easy and consider all options of choosing the corresponding dominating $P_4$s. This leads to $2^kO(n^{4k})$ branches, which is polynomial as $k$ is fixed.
We discard those options that do not result in mutually induced connected subgraphs $D^i$ with $Z^i\subseteq V(D^i)$ and also those with an edge between a vertex of some guessed $D^i$ and some $Z_j$ that will correspond to a difficult subgraph $D^j$.

For each remaining branch, we delete the vertices of each easy $D^i$ and all their neighbours from $G$. We also remove the corresponding sets $Z_i$ from ${\cal Z}$. For simplicity we denote the new instance by $(G,{\cal Z})$ again, and we let $|{\cal Z}|=k$. If $k\leq 1$, then we can solve the problem in a trivial way. Suppose that $k\geq 2$.
We note that $G$ is still $P_6$-free; every $Z_i\in {\cal Z}$ still has size at least~$2$ and that the union of the sets in ${\cal Z}$ is an independent set such that no two vertices from different $Z_i$ and $Z_j$ have a common neighbour. Moreover, if $(G,{\cal Z})$ has a solution, then every connected subgraph in it is difficult.

Consider a solution $(D^1,\ldots,D^k)$ (if it exists). As each $D^i$ is difficult, it contains (by definition) a connected dominating set~$X_i$ such that $G[X_i]$ is $P_4$-free. By Theorem~\ref{t-cs} (take $r=4$),  every $G[X_i]$ has a connected dominating set 
$Y_i$ of size at most~$2$. We consider all options of choosing the connected sets $Y_i$. For every chosen $Y_i$ of size~$2$, we contract the edge between the two vertices of $Y_i$. The resulting graph is still $P_6$-free by Lemma~\ref{l-contract}. Note that the number of branches is $O(n^{2k})$, which is polynomial as $k$ is fixed. In each branch, we may now assume that $Y_i=\{y_i\}$ for $i\in \{1,\ldots,k\}$ and thus $y_i$ will dominate $G[X_i]$. We discard an option if $\{y_1,\ldots,y_k\}$ is not an independent set, or if some $y_i$ is adjacent to a vertex of some $Z_j$ with $i\neq j$. For every other branch we continue as follows. From $G$ we first delete 

\begin{itemize}
\item [(ii)] for every $i\in \{1,\ldots,k\}$, every neighbour of $y_i$ adjacent to a vertex of $Z_j \cup \{y_j\}$ for some $j\neq i$.
\end{itemize}
This takes polynomial time and we may do this, as these common neighbours cannot belong to any subgraph in a solution $(D^1,\ldots,D^k)$ with $y_i\in V(D^i)$ for $i\in \{1,\ldots,k\}$.

We may assume without loss of generality that in the solution $(D^1,\ldots,D^k)$ we are looking for the subgraphs are minimal. That is, we cannot replace a subgraph $D^i$ with some subgraph $F^i$ with $V(F^i)\subset V(D^i)$.
So each vertex~$x$ of $X_i\setminus \{y_i\}$ is either a vertex of $Z_i$ or has a neighbour~$z_i$ in $Z_i$ that is not adjacent to any vertex of $X_i\setminus \{x_i\}$; in the latter case we say that $z_i$ is a {\it private} neighbour of $x_i$. For every $i\in \{1,\ldots,k\}$, at least one such pair $(x_i,z_i)$ exists, as otherwise $y_i$ dominates $D^i$, and thus $D^i$ would not be difficult.

We now consider for $i\in \{1,\ldots,k\}$, all options of choosing the pairs $(x_i,z_i)$ where $x_i$ is a neighbour of $y_i$ that is not from $Z_i$, whilst $z_i$ is taken from $Z_i$. We discard those options where $x_i$ and $z_i$ are not adjacent or where $z_i$ is adjacent to $y_i$; in both cases $z_i$ will not be a private neighbour of $x_i$. We also discard every option where the graphs $G[\{y_i,x_i,z_i\}]$ are not mutually induced. Note that the number of branches is $O(n^{2k})$ (so polynomial). For each branch that we have not discarded we continue by first deleting from $G$, the following sets of vertices:

\begin{itemize}
\item [(iii)] for $i\in \{1,\ldots,k\}$, every neighbour of $z_i$ not equal to $x_i$;
\item [(iv)] for $i\in \{1,\ldots,k\}$, every neighbour of $x_i$ that is adjacent to a vertex of $Z_j \cup \{x_j,y_j\}$ for some $j\neq i$;
\end{itemize}
This takes polynomial time. Moreover, we are allowed to delete all these vertices, as none of them can be used in the solution that we are trying to construct. In particular, every~$z_i$ will be a private neighbour of $x_i$, so every $z_i$ has only one neighbour in $V(D^1)\cup \cdots \cup V(D^k)$.
We now use the fact that (i)--(iv) hold to prove the following claim: 

\medskip
\noindent
{\it A solution for this branch exists if and only if $N_G[y_i]$ dominates $Z_i$ for every $i\in \{1,\ldots,k\}$.} 

\medskip
\noindent
First suppose that $(G,{\cal Z})$ has a solution $(D^1,\ldots,D^k)$ such that for every $i\in \{1,\ldots,k\}$ it holds that $y_i$ dominates $X_i$.
Then $N_G[y_i]$ contains $X_i$, which dominates $Z_i$ by definition.

Now suppose that $N_G[y_i]$ dominates $Z_i$ for every $i\in \{1,\ldots,k\}$. We claim that the $k$-tuple $(D^1,\ldots,D^k)$ where $V(D^i)=N_G[y_i] \cup Z_i$ for every $i\in \{1,\ldots,k\}$ is a solution. First note that each $D^i$ is connected (as $N_G[y_i]$ dominates $Z_i$)
and contains $Z_i$. Hence, it remains to show that $D^1,\ldots,D^k$ are mutually induced.

For a contradiction, let $u_i \in V(D^i)$ and $u_j \in V(D^j)$ be adjacent for some $i\neq j$. Due to (i)--(iv), we find that
$u_i$ belongs to $N_G[y_i]\setminus \{x_i\}$ and $u_j$ belongs to $N_G[y_j]\setminus \{x_j\}$.
First suppose that $u_i$ is not adjacent to $x_i$ or $u_j$ is not adjacent to $x_j$, say $u_i$ is not adjacent to $x_i$.
Now, $\{z_i, x_i, y_i, u_i, u_j,y_j\}$ induces a $P_6$, a contradiction. Hence, $u_i$ must be adjacent to $x_i$ and $u_j$ must be adjacent to $x_j$. However, now $\{z_i, x_i, u_i, u_j, x_j, z_j\}$ induces a $P_6$, another contradiction. Hence, we have proven the claim.

\medskip
\noindent
We can check in polynomial time whether $N_G[y_i]$ dominates $Z_i$ for every $i\in \{1,\ldots,k\}$. By the above claim, this means that we can check in polynomial time if a certain branch leads to a solution. As the total number of branches is polynomial, the running time of our algorithm is polynomial.
The correctness of our algorithm follows from its description; note that we examined all possible situations.
\qed
\end{proof}

\subsection{Two More Algorithmic Results}

In this section we present our final two polynomial-time algorithms. The first result holds for fixed~$k$. The second result is for a smaller graph class but holds even when $k$ and $\ell$ are both part of the input. To prove the second result, we will use the algorithm of the first result as a subroutine.

\begin{lemma}\label{l-2p4}
For every $k\geq 2$ and $s\geq 0$, {\sc $k$-Induced Disjoint Connected Subgraphs} is polynomial-time solvable for $(sP_1+2P_4)$-free graphs.
\end{lemma}

\begin{proof}
Let $s\geq 0$. Let $(G,{\cal Z})$ be an instance of {\sc $k$-Induced Disjoint Connected Subgraphs}, where $G$ is an $(sP_1+2P_4)$-free graph  and ${\cal Z}=\{Z_1,\ldots,Z_k\}$. By Lemma~\ref{l-independent}, we may assume that every $Z_i\in {\cal Z}$ has size at least~$2$ and that the union of the sets in ${\cal Z}$ is an independent set.

By Lemma~\ref{l-useful} we may assume without loss of generality that for a solution $(D^1,\ldots,D^k)$ it holds that every $D^i$ has size at most
$(2s+15)|Z_i|$. Call a set $Z_i$ {\it small} if $|Z_i|\leq s-1$, else $Z_i$ is {\it large}. 

If $Z_i$ is small, then $D^i$  has size at most $t$ where $t=(2s+15)(s-1)$. Let ${\cal Z}'$ be the subset of ${\cal Z}$ that contains the small sets of ${\cal Z}$. We consider all $O(n^{kt})$ options of choosing the corresponding connected subgraphs $D^i$. We check in polynomial time if there are no forbidden edges between these subgraphs or between such a subgraph and a large set, and we also check if each such $D^i$ contains $Z_i$. If one of the conditions is violated, we discard the choice. Otherwise, we delete the vertices of $N[V(D^i)]$ from $G$ for each small $Z_i$ and we also delete the small sets from ${\cal Z}$. Note that the resulting graph is still $(sP_1+2P_4)$-free.
For simplicity, we denote the resulting instance by $(G,{\cal Z})$ again. If $|{\cal Z}|=k\leq 1$, we solve the instance directly in a trivial way. Keep all created instances with more than one set in~${\cal Z}$.
Note that we created $O(n^{kt})$ branches in this way and the instance of each branch that we kept has no small sets and $k\geq 2$.

We say that a subgraph $D^i$ in a solution is {\it easy} if $D^i$ is $P_4$-free; otherwise $D^i$ is {\it difficult}.
As each $Z_i$ in ${\cal Z}$ is now large, each $D^i$ contains an induced $sP_1$. As $G$ is $(sP_1+2P_4)$-free, this means that at least $k-2$ subgraphs of any solution are easy. We consider all $O(k^2)$ options of choosing the easy subgraphs.
By Theorem~\ref{t-cs} (take $r=4$) we find that each easy $D^i$ has a connected dominating set $X_i$ of size at most~$2$. We consider all $O(n^{2(k-2)})$ options of choosing the vertices of the sets $X_i$ corresponding to the easy subgraphs $D^i$. Again, we discard a choice if some guessed $D^i=G[X_i\cup Z_i]$ is not connected or a vertex of some guessed $D^i$ is adjacent to a vertex of another guessed $D^h$ or to a vertex of a set $Z_j$ that will be contained in a $D^j$ that is difficult. Otherwise, we obtained, in polynomial time, a new instance after deleting the vertices of $N[V(D^i)]$ from $G$ for each easy $D^i$ and deleting the corresponding sets from ${\cal Z}$. For simplicity, we denote the resulting instance by $(G,{\cal Z})$ again. Observe that $G$ is still $(sP_1+2P_4)$-free, and in particular that ${\cal Z}$ has size $k\leq 2$. If $k\leq 1$, the problem has become trivial to solve. Assume that $k=2$. Then, for simplicity, we write ${\cal Z}=\{Z_1,Z_2\}$; note both $Z_1$ and $Z_2$ are large, as these sets were large at the start of the initial branch.
 
 So, to summarize, we have created in polynomial time $O(k^2n^{2(k-2)})$ branches and for each branch we have an instance $(G,{\cal Z})$ of {\sc $2$-Induced Disjoint Connected Subgraphs} where $G$ is $(sP_1+2P_4)$-free and ${\cal Z}=\{Z_1,Z_2\}$. Moreover, if an instance has a solution $(D^1,D^2)$ then both $D^1$ and $D^2$ are difficult.
It remains to show how we can solve this problem in polynomial time for each created instance $(G,{\cal Z})$. We do this below.

Consider a created instance $(G,{\cal Z})$. Let $(D^1,D^2)$ be a solution for it. As $D^1$ is difficult, $D^1$ has an induced $P_4$ by definition. As $G$ is $(sP_1+2P_4)$-free, this means that $D^2$ is $(sP_1+P_4)$-free. As $D^2$ is difficult, $D^2$ has an induced $P_4$ as well, say on vertices $a,b,c,d$. As $D^2$ is $(sP_1+P_4)$-free and $Z_2$ is an independent set, the set $\{a,b,c,d\}$ must dominate all but at most $s-1$ vertices of $Z_2$. Let $Z_2^*$ be the subset of $Z_2$ that consists of vertices not adjacent to any vertex of $\{a,b,c,d\}$; so, $Z_2^*$ has size at most~$s-1$.

Fix a vertex $u$ of $D^2$. As $D^2$ is $(sP_1+P_4)$-free and thus $P_{2s+3}$-free, there exists a path $P_i$ of length at most $2s+1$ from each $z_i\in Z_2^*$ to $u$. For the same reason we can also choose a path $P_d$ from $d$ to $u$ of length at most $2s+3$. Then replacing $D^2$ by the subgraph $F^2$ of $D^2$ that is induced by the union of $\{a,b,c,d,u\}\cup Z_2$ and the inner vertices of all these paths leads to an alternative solution $(D^1,F^2)$, where $F^2-Z_2$ has size at most $t'$, for $t'=(2s+2)(s+1)$.

It now remains to consider all $O(n^{t'})$ options of choosing the vertices of $F^2-Z_2$. For each option we check if $F^2$ is connected and contains $Z_2$ and if $G-N[F^2]$ contains a connected component that contains $Z_1$. This can be done in polynomial time.

Correctness of our algorithm follows from the above description. As the number of branches is polynomial and each branch can be processed in polynomial time, the running time of our algorithm is polynomial.
\qed
\end{proof}

\noindent
As mentioned, we use the algorithm of Lemma~\ref{l-2p4} as a subroutine in our next result, which holds even when $k$ and $\ell$ are part of the input. In addition, we also use the algorithm of of Lemma~\ref{l-p5} as a subroutine.

\begin{lemma}\label{l-sp1c}
For every $s\geq 0$, {\sc Induced Disjoint Connected Subgraphs} is polynomial-time solvable for $(sP_1+P_3+P_4)$-free graphs.
\end{lemma}

\begin{proof}
Let $(G,{\cal Z})$ be an instance of  {\sc Induced Disjoint Connected Subgraphs}, where $G$ is an $(sP_1+P_3+P_4)$-free graph for some integer $s\geq 0$ and ${\cal Z}=\{Z_1,\ldots,Z_k\}$. 
We assume without loss of generality that $|Z_1|\geq |Z_2|\geq \cdots \geq |Z_k|$.
By Lemma~\ref{l-independent}, we may assume that $k\geq 2$; every $Z_i'\in {\cal Z}$ has size at least~$2$; and the union of the sets in ${\cal Z}$ is an independent set. 

\medskip
\noindent
{\bf Case 1.} $|Z_k|\leq s-1$.\\
By Lemma~\ref{l-useful} we may assume without loss of generality that for a solution $(D^1,\ldots,D^k)$, we have $|V(D^k)|\leq t$, where
$t=(2s+13)(s-1)$. We consider all $O(n^t)$ options of choosing the subgraph $D^k$. 
For each choice we check in polynomial time if $D^k$ is connected, contains
$Z_k$ and that no vertex of $D^k$ is adjacent to any vertex of $Z_1\cup \ldots \cup Z_{k-1}$. If one of these conditions does not hold, then we discard the choice. Otherwise, we solve  {\sc Induced Disjoint Connected Subgraphs} on instance $(G-N[V(D^k)], {\cal Z}\setminus \{Z_k\})$. This can be done in polynomial time. Namely, the graph $D^k$ contains an induced $P_3$, as $Z_k$ is an independent set of size at least~$2$. Thus, $G-N[D^k]$ is $(sP_1+P_4)$-free (and hence, $(sP_1+P_5)$-free) and we can apply Lemma~\ref{l-p5}.
As the number of branches is polynomial as well, the running time for processing Case~1 is polynomial. 

\medskip
\noindent
{\bf Case 2.} $|Z_k|\geq s$ (and hence $|Z_i|\geq s$ for every $i$).\\
We need to make a distinction into two more subcases (recall that $k\geq 2$).

\medskip
\noindent
{\bf Case 2a.} $k=2$.\\
As $G$ is $(sP_1+P_3+P_4)$-free, $G$ is also $(sP_1+2P_4)$-free and we can use Lemma~\ref{l-2p4} to process Case~2a in polynomial time.

\medskip
\noindent
{\bf Case 2b.} $k\geq 3$.\\
Let $(D^1,\ldots,D^k)$ be a solution (assuming it exists). As $Z_2$ and $Z_3$ are independent sets of size at least~$s$, both $D^2$ and $D^3$ contain an induced $sP_1$ and an induced~$P_3$. Hence, $D^1$ is $P_4$-free. By Theorem~\ref{t-cs} (take $r=4$), we find that $D^1$ has a connected dominating set $X_1$ of size at most~$2$. We consider all $O(n^2)$ options of choosing $X_1$. We discard a choice if $D^1=G[X_1\cup Z_1]$ is not connected or a vertex of $D^1$ is adjacent to a vertex of $Z_2\cup \cdots \cup Z_k$. For each non-discarded choice, we solve {\sc Induced Disjoint Connected Subgraphs} on instance $(G-N[V(D^1)],{\cal Z}\setminus \{Z_1\})$. The latter takes polynomial time, as we can apply Lemma~\ref{l-p5}: as $D^1$ has an induced~$P_3$, the graph $G-N[V(D^1)]$ is $(sP_1+P_4)$-free and thus $(sP_1+P_5)$-free. As the number of branches is polynomial as well, the running time for processing Case~2b is polynomial.

\medskip
\noindent
The correctness of our algorithm follows from the case descriptions. As each of the cases can be done in polynomial time, the total running time of our algorithm is polynomial.
\qed
\end{proof}

\section{NP-Completeness Results}\label{s-np}

In this section we show a number of NP-completeness results that we need for proving our main theorems. Some of these results hold even for more restricted graph classes.
If $\ell=2$, we write {\sc Induced Disjoint Paths} instead of {\sc Induced Disjoint Connected $\ell$-Subgraphs}. 

\subsection{High Girth}

The {\it girth} of a graph~$G$ that is not a forest is the length of a shortest cycle of $G$.
We prove two results for graphs of high girth.

\begin{lemma}\label{l-girth}
For every $g \geq 3$, {\sc Induced Disjoint Paths} is \NP-complete for the class of graphs of girth at least~$g$.
\end{lemma}

\begin{proof}
We reduce from {\sc Disjoint Paths}, which is known to be \NP-complete for graphs of girth at least $g$ for every $g \geq 3$~\cite[Lemma~9]{KMPSV22}. We observe that the reduction of~\cite{KMPSV22} is from {\sc Disjoint Paths} on arbitrary graphs by subdividing all edges an appropriate number of times, say $\lceil g/3 \rceil \geq 1$ times.  We note that for {\sc Disjoint Paths}, the reduction of Lynch~\cite{Ly75} has the property that the terminals are on disjoint vertices. Then the construction of~\cite{KMPSV22} guarantees that the terminals in the constructed graph are disjoint. Consider such an instance of {\sc Disjoint Paths} on a graph $G$ and terminal pairs $(s_1,t_1),\ldots,(s_k,t_k)$. By the previous, $s_1,\ldots,s_k,t_1,\ldots,t_k$ are disjoint. Since the construction subdivides each edge at least once, we obtain an equivalent instance of {\sc Induced Disjoint Paths} on a graph of girth at least~$g$.
\qed
\end{proof}

\noindent
For some of our other results we prove \NP-hardness by reducing from the {\sc $2$-Disjoint Connected Subgraphs} problem.
Recall that this problem asks if a given graph has two vertex-disjoint connected subgraphs containing pre-specified sets of vertices $Z_1$ and $Z_2$, respectively. 

\begin{lemma}\label{l-girth2}
For every $g \geq 3$, {\sc $2$-Induced Disjoint Connected Subgraphs} is \NP-complete for the class of graphs of girth at least~$g$.
\end{lemma}

\begin{proof}
We reduce from {\sc $2$-Disjoint Connected Subgraphs}, which is known to be \NP-complete for graphs of girth at least $g$ for every $g \geq 3$~\cite[Lemma~6]{KMPSV22}. Again, the reduction of~\cite{KMPSV22} subdivides the edges of an instance of {\sc $2$-Disjoint Connected Subgraphs} on general graphs, and we may assume that it does so at least once. Then we obtain an equivalent instance of {\sc $2$-Induced Disjoint Connected Subgraphs} on a graph of girth at least~$g$.
\qed
\end{proof}

\subsection{Line Graphs}

The {\it line graph} $L(G)$ of a graph $G$ has vertex set $\{v_e \mid e \in E(G)\}$ and an edge between $v_e$ and $v_f$ if and only if $e$ and $f$ are incident on the same vertex in $G$. 

The following two lemmas show \NP-completeness for line graphs. Lemma~\ref{l-line} is due to Fiala et al.~\cite{FKLP12}. They consider a more general variant of {\sc Induced Disjoint Paths}, but their reduction holds in our setting as well.
Lemma~\ref{l-line2} can be derived from the \NP-completeness of {\sc $2$-Disjoint Connected Subgraphs}~\cite{HPW09}.

\begin{lemma}[\cite{FKLP12}]\label{l-line}
{\sc Induced Disjoint Paths} is \NP-complete for the class of line graphs.
\end{lemma}

\begin{proof}
Fiala et al.~\cite[Theorem~24]{FKLP12} prove that {\sc Induced Disjoint Paths} is \NP-complete for line graphs by reducing from {\sc Disjoint Paths} on general graphs. However, in the paper, they consider a more flexible variant where (among others) terminals can be adjacent. Fortunately, this extra freedom is not used in the reduction. We note that for {\sc Disjoint Paths}, the reduction of Lynch~\cite{Ly75} guarantees that the terminals are on disjoint vertices. Then the construction of~\cite{FKLP12} guarantees that the terminals in the constructed graph form an independent set and hardness for our variant follows.
\qed
\end{proof}

\begin{lemma}\label{l-line2}
{\sc $2$-Induced Disjoint Connected Subgraphs} is \NP-complete for the class of line graphs.
\end{lemma}

\begin{proof}
We reduce from {\sc $2$-Disjoint Connected Subgraphs}, which is known to be \NP-complete~\cite{HPW09}. 
We describe a reduction that is similar to Fiala et al.~\cite[Theorem~24]{FKLP12} for {\sc Induced Disjoint Paths}.

Let $(G, Z_1, Z_2)$ be an instance of {\sc $2$-Disjoint Connected Subgraphs}. For each vertex $z \in Z_1 \cup Z_2$, create a new vertex $v_z$ and connect it by an edge $e_z$ to $z$. Let $G'$ denote the new graph. Note that $(G,Z_1,Z_2)$ is a yes-instance of {\sc $2$-Disjoint Connected Subgraphs} if and only if $(G', \{v_z \mid z \in Z_1\}, \{v_z \mid z \in Z_2\})$ is a yes-instance. Now consider the line graph $L(G')$. For each $z \in Z_1 \cup Z_2$, let $v(e_z)$ be the vertex in $L(G')$ corresponding to $e_z$. Then it can be readily seen that $(G,Z_1,Z_2)$ is a yes-instance of {\sc $2$-Disjoint Connected Subgraphs} if and only if $(L(G'), \{v(e_z) \mid z \in Z_1\}, \{v(e_z) \mid z \in Z_2\})$ is a yes-instance of {\sc $2$-Induced Disjoint Connected Subgraphs}.
\qed
\end{proof}

\subsection{Forbidding Some Linear Forest}

Finally, we show two lemmas for graphs without certain induced linear forests. Lemma~\ref{l-3p2p7} shows that {\sc $2$-Induced Disjoint Connected Subgraphs} is \NP-complete for $(3P_2,P_7)$-free graphs. It is readily seen that the gadget constructed in the hardness reduction is not $2P_4$-free. Note that this is in line with Theorem~\ref{thm:k-IDCS}.
However, Lemma~\ref{l-2p4npc} shows that \NP-completeness does hold for $2P_4$-free graphs when the number $k$ of terminal sets is part of the input. That is, {\sc Induced Disjoint Connected Subgraphs} is \NP-complete for $2P_4$-free graphs.

\begin{figure}[tbp]
\[
\xymatrix{
C' &  \bullet \ar@{-}[dr] & \bullet & \cdots & \bullet & \bullet \ar@{-}[dl] \\
x' &  \bullet \ar@{=}[rrr] \ar@{-}[d]  & \bullet \ar@{-}[d] & \cdots & \bullet \ar@{-}[d] & \\
x &   \bullet \ar@{=}[rrr]  & \bullet & \cdots & \bullet & \\
C &   \bullet \ar@{-}[ur] & \bullet & \cdots & \bullet & \bullet \ar@{-}[ul] \\
}
\]
\caption{Connections between cliques in the proof of Lemma~\ref{l-3p2p7}. The horizontal double lines indicate these vertices are joined in a clique. In the diagram we show variable $x_2$ appearing in clause $C_1$ and variable $x_n$ appearing in clause $C_m$.}
\label{fig:disjoint-induced-paths-hardness}
\end{figure}
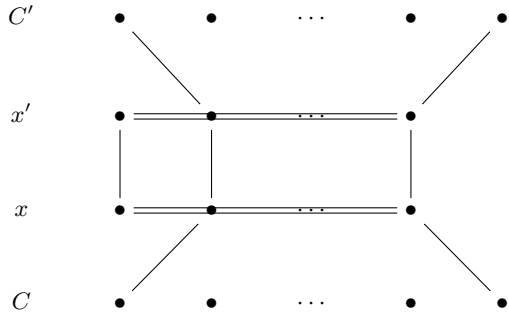

\begin{lemma}\label{l-3p2p7}
{\sc $2$-Induced Disjoint Connected Subgraphs} is \NP-complete for the class of ($3P_2,P_7)$-free graphs.
\end{lemma}

\begin{proof}
We reduce from {\sc Not-All-Equal-$3$-Sat}, known to be NP-complete~\cite{Sc78}. Let $({\cal X},{\cal C})$ be an instance of {\sc Not-All-Equal-$3$-Sat} containing $n$ variables $x_1,\ldots,x_n$ and $m$ clauses $C_1,\ldots,C_m$. We construct a graph $G$ as follows. Let $X$ be a clique of size $n$ on vertices $v_1,\ldots,v_n$. Introduce a copy $v'_i$ of each $v_i$ in $X$. Call the new set $X'$ and make it a clique. Add the edges $v_iv_i'$ for each $v_i$ in $X$. Let $C$ be an independent set of size $m$ on vertices $c_1,\ldots,c_m$. Introduce a copy $c'_j$ of each vertex $c_j$ in $C$. Call the new set $C'$ (and keep it an independent set). Now for all $1 \leq i \leq n$ and $1 \leq j \leq m$, add an edge $v_ic_j$ and an edge $v_i'c_j'$ if clause $C_j$ contains variable $x_i$. Set $Z_1=C$ and $Z_2=C'$. Then, $(G,Z_1,Z_2)$ is an instance of {\sc $2$-Induced Disjoint Connected Subgraphs}. See Figure~\ref{fig:disjoint-induced-paths-hardness}.

Observe that $G$ is $P_7$-free. Indeed, let $P$ be any longest induced path in $G$. Then $P$ can contain at most two vertices from $X$ and at most two vertices from $X'$. If $P$ contains at most one vertex from $C$ and at most one vertex from $C'$, then $P$ has length at most $2 + 2 + 1+ 1 = 6$.  On the other hand, if $P$ contains two vertices from $C$ or two vertices from $C'$, then $P$ has length at most $3$.

We also observe that $G$ is $3P_2$-free, as any $P_2$ must contain at least one vertex from $X$ or from $X'$, and $X$ and $X'$ are cliques.
So we are done after proving the following claim: $({\cal X},{\cal C})$ is a yes-instance of {\sc Not-All-Equal-$3$-Sat} if and only if $(G,Z_1,Z_2)$ is a yes-instance of {\sc $2$-Induced Disjoint Connected Subgraphs}.

In the forward direction, let $\tau$ be a satisfying truth assignment. We put in $A$ every vertex of $X$ for which the corresponding variable is set to true. We put in $A'$ every vertex of $X'$ for which the corresponding variable is set to false. 
As each clause $C_j$ contains at least one true variable, $c_j$ is adjacent to a vertex in $A$. Similarly, each clause $C_j$ contains at least one false variable, so each $c'_j$ is adjacent to a vertex in $A'$. 
As $X$ and $X'$ are cliques, $A$ and $A'$ are cliques. Hence, $G[C \cup A]$ and $G[C' \cup A']$ are connected. 

Now suppose there is an edge between a vertex of $C \cup A$ and a vertex of $C' \cup A'$. Then, by construction, this edge must be equal to some $v_iv_i'$, which means that $v_i$ is in $A$ and $v_i'$ is in $A'$, so $x_i$ must be true and false at the same time, a contradiction. Hence, there exists no edge between a vertex from $C \cup A$ and a vertex from $C' \cup A'$. We conclude that $(C \cup A, C' \cup A')$ is a solution.

In the backwards direction, let $(C \cup A, C' \cup A')$ be a solution. Then, by definition, there is no edge between $C \cup A$ and $C' \cup A'$, which means that there is no edge between $A$ and $A'$. Then $A \subseteq X$ and $A' \subseteq X'$, since $X$ and $X'$ are cliques and $A$ ($A'$) needs to contain at least one vertex of $X$ ($X'$). Also, there is no variable $x_i$ such that $v_i$ is in $A$ and $v'_i$ is in $A'$. This means we can define a truth assignment $\tau$ by setting all variables corresponding to vertices in $A$ to be true, all variables corresponding to vertices in $A'$ to be false, and all remaining vertices in ${\cal X}$ to be true (or false, it does not matter).

As $C$ is an independent set and $C \cup A$ is connected, each $c_j$ has a neighbour in $A$. So each $C_j$ contains a true literal. As $C'$ is an independent set and $C' \cup A'$ is connected, each $c'_j$ has a neighbour in $A'$. So each $C_j$ contains a false literal. Hence, $\tau$ is a satisfying truth assignment. This completes the proof.
\qed
\end{proof}

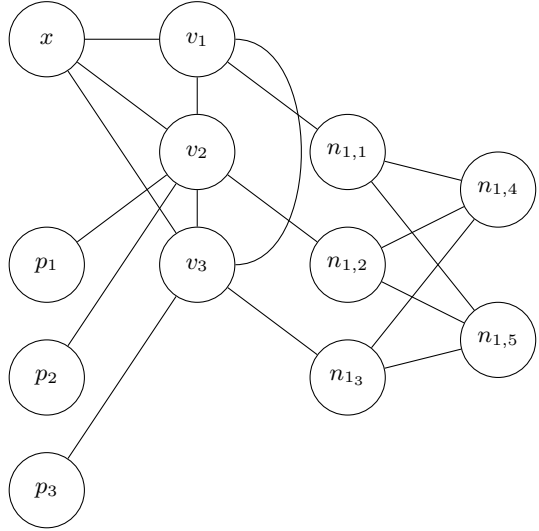
\begin{figure} [tbp]
\begin{center}
\begin{tikzpicture}[main_node/.style={circle,draw,minimum size=1cm,inner sep=3pt]}]
\node[main_node](x) at (0,0){$x$};
\node[main_node](p1) at (0, -3){$p_1$};
\node[main_node](p2) at (0,-4.5){$p_2$};
\node[main_node](p3) at (0,-6){$p_3$};
\node[main_node](v1) at (2,0){$v_1$};
\node[main_node](v2) at (2,-1.5){$v_2$};
\node[main_node](v3) at (2,-3){$v_3$};
\node[main_node](n11) at (4,-1.5){$n_{1,1}$};
\node[main_node](n12) at (4,-3){$n_{1,2}$};
\node[main_node](n13) at (4,-4.5){$n_{1_3}$};
\node[main_node](n14) at (6, -2){$n_{1,4}$};
\node[main_node](n15) at (6, -4){$n_{1,5}$};
\draw(n11)--(v1);
\draw(n12)--(v2);
\draw(n13)--(v3);
\draw(n13)--(n14);
\draw(n13)--(n15);
\draw(n12)--(n14);
\draw(n12)--(n15);
\draw(n11)--(n14);
\draw(n11)--(n15);
\draw(v2)--(p1);
\draw(v2)--(p2);
\draw(v3)--(p3);
\draw(x)--(v1);
\draw(x)--(v2);
\draw(x)--(v3);
\draw(v1)--(v2)--(v3);
\draw(v1) [out=0, in=0]to (v3);
\end{tikzpicture}
\end{center}
\caption{The graph constructed in the proof of Lemma~\ref{l-2p4npc} corresponding to an instance of {\sc Monotone $3$-Satisfiability} with three positive clauses and one negative clause $ (\neg x_1 \lor \neg x_2 \lor \neg x_3)$.}\label{f-label}
\end{figure}

\begin{lemma}\label{l-2p4npc}
{\sc Induced Disjoint Connected Subgraphs} is \NP-complete for the class of $2P_4$-free graphs.
\end{lemma}

\begin{proof}
We reduce from {\sc Monotone $3$-Satisfiability}~\cite{Li97}. Let $\Phi$ be an instance of {\sc Monotone $3$-Satisfiability} with $n$ variables $v_1, \dots ,v_n$, $l$  clauses with only positive literals $P_1, \dots, P_l$ and $m$ clauses with only negated literals $N_1, \dots, N_m$. We define an instance $(G, Z_1, \dots Z_{m+1})$ of {\sc Induced Disjoint Connected Subgraphs} as follows (see also Fig.~\ref{f-label}).

\begin{itemize}
	\item Add a clique, with one vertex $v_i$ for each variable.
	
	\item Add an independent set consisting of one vertex $p_i$ for each positive clause together with one further vertex $x$ adjacent to every variable vertex.
	
	\item Add an edge between each vertex $p_i$ and the variable vertices contained in the corresponding clause $P_i$.
	
	\item For each negative clause add a complete bipartite graph $K_{2,3}$ where the vertices contained in the part of size~$3$, $n_{i,1}$, $n_{i,2}$ and $n_{i,3}$, represent the literals of the clause $N_i$, whilst the vertices contained in the part of size~$2$ are denoted as $n_{i,4}$ and $n_{i,5}$.
	
	\item Add edges from each literal vertex $n_{i,j}, 1 \leq j \leq 3$ to the corresponding variable vertex.
	
	\item Let $Z_1$ consist of each positive clause vertex $p_i$ together with the vertex $x$.
	
	\item Let $Z_i$ consist of the two vertices $n_{i-1,4}$ and $n_{i-1,5}$ for $2 \leq i \leq m+1$.
\end{itemize}
	
We first show that $G$ is $2P_4$-free. Note that at most one of the two paths in an induced $2P_4$ contains any variable vertex. Since every neighbour of the vertices $\{p_1 \dots p_l, x\}$ is a variable vertex, none of these vertices is contained in an induced $P_4$ which excludes variable vertices. As the complete bipartite graph $K_{2,3}$ is $P_4$-free, this implies that $G$ is $2P_4$-free.
		
	Next we show that $G$ is a yes-instance of {\sc Induced Disjoint Connected Subgraphs} if and only if $\Phi$ is a yes-instance of {\sc Monotone 3-satisfiability}. Given a satisfying assignment of $\Phi$, let $S_1=Z_1 \cup T$ where $T$ is the set of variable vertices corresponding to true variables. For $2 \leq i \leq m+1$, let $S_i=Z_i \cup F_i$ where $F_i$ is the set of literal vertices $n_{{i-1},j}$ adjacent to variable vertices appearing in $N_{i-1}$ which are assigned to be false. Note that no subgraph $S_i, 2 \leq i \leq m+1$ contains a variable vertex. $S_1$ is connected since at least one variable appearing in each positive clause must be true in any satisfying assignment. Similarly each $S_i$ is connected for $2 \leq i \leq m+1$ since any negative clause must contain at least one variable which is assigned to be false. Any edge between $S_i$ and $S_j$ for $ i \neq j$ must contain a variable vertex since the remaining edges are those contained in copies of $K_{2,3}$ and hence either have two endpoints in the same subgraph $S_i$ or one endpoint  contained in no subgraph $S_i$. Therefore we may assume that $i=1$. If a variable vertex $v_i$ is contained in $S_1$ and has a neighbour in a second subgraph $S_j$ it must be both true and false in a satisfying assignment, a contradiction.
	
	 If $(G, Z_1 \dots Z_{m+1})$ is a yes-instance of {\sc Induced Disjoint Connected Subgraphs}, consider any solution $(S_1,\ldots,S_{m+1})$ such that $Z_i \subseteq S_i$ for $1 \leq i \leq m+1$. Note that $\{n_{i-1,1}, n_{i-1,2}, n_{i-1,3} \mid 2 \leq i \leq m+1\} \subseteq N(Z_2 \cup \cdots \cup Z_{m+1})$, and thus $S_1 \subseteq Z_1 \cup \{v_1,\ldots,v_n\}$. Since the variable vertices form a clique and $S_1$ will need to contain at least one variable vertex, $(S_2 \cup \cdots \cup S_{m+1}) \cap \{v_1,\ldots,v_n\} = \emptyset$. Set each variable whose corresponding vertex is contained in $S_1$ to true and each remaining variable to false. We claim this yields a satisfying assignment for $\Phi$. For a positive clause $P_i$, note that $S_1$ must connect $p_i$ to $x$, which is only possible if a variable vertex adjacent to $p_i$ is in $S_1$. This variable is contained in the clause and set to true, and will thus satisfy the clause. For a negative clause $N_{i-1}$ with $2 \leq i \leq m+1$, we note that $Z_i = \{n_{i-1,4}, n_{i-1,5}\}$ is connected by $S_i$, which is only possible if one of $n_{i-1,1}, n_{i-1,2}, n_{i-1,3}$ is in $S_i$, say $n_{i-1,1}$. But then the variable vertex corresponding to the first literal of the clause cannot be in $S_1$, and thus is set to false and satisfies the clause.
\qed\end{proof}

\section{The Proofs of Theorems~\ref{thm:IDPnew}--\ref{thm:k-IDCS}}\label{s-proofs}

We are now ready to prove Theorems~\ref{thm:IDPnew}--\ref{thm:k-IDCS}, which we restate below.

\medskip
\noindent
{\bf Theorem~\ref{thm:IDPnew} (restated).}
{\it Let $\ell\geq 2$. For a graph $H$, {\sc Induced Disjoint Connected $\ell$-Subgraphs} on $H$-free graphs is polynomial-time solvable if  
$H \ssi sP_3+P_6$ for some $s\geq 0$; \NP-complete if $H$ is not a linear forest; and quasipolynomial-time solvable otherwise.}

\begin{proof}
We prove the theorem for $\ell=2$; extending the proof to $\ell\geq 3$ is trivial.
If $H$ contains a cycle $C_s$, then we use Lemma~\ref{l-girth} by setting the girth to $g=s+1$.
Suppose that $H$ contains no cycle, that is, $H$ is a forest.
If $H$ contains a vertex of degree at least~$3$, then we use Lemma~\ref{l-line}, as in that case the class of $H$-free graphs contains the class of $K_{1,3}$-free graphs, which in turn contains the class of line graphs.
In the remaining cases, $H$ is a linear forest.
If $H\ssi sP_3+P_6$ for some $s\geq 0$ we use Lemma~\ref{l-sp6}.
Else we use Lemma~\ref{l-quasi}. \qed
\end{proof}

\noindent
{\bf Theorem~\ref{thm:IDCS} (restated).}
{\it For a graph $H$ such that $H\neq sP_1+P_6$ for some $s\geq 0$, {\sc Induced Disjoint Connected Subgraphs} on $H$-free graphs
is polynomial-time solvable for $H$-free graphs if $H \ssi sP_1+P_3+P_4$ or $H \ssi sP_1+ P_5$ for some $s\geq 0$, and it is \NP-complete otherwise.}

\begin{proof}
If $H$ is not a linear forest, we use Theorem~\ref{thm:IDPnew}. Suppose $H$ is a linear forest.
If $H\ssi sP_1+P_5$ for some $s\geq 0$ we use Lemma~\ref{l-p5}.
If $H\ssi sP_1+P_3+P_4$ for some $s\geq 0$ we use Lemma~\ref{l-sp1c}.
If $3P_2\ssi H$ or $P_7\ssi H$ we use Lemma~\ref{l-3p2p7}.
Otherwise $2P_4\ssi H$ and we use Lemma~\ref{l-2p4npc}. \qed
\end{proof}

\noindent
{\bf  Theorem~\ref{thm:k-IDCS} (restated).}
{\it Let $k\geq 2$. For a graph $H$, {\sc $k$-Induced Disjoint Connected Subgraphs} on $H$-free graphs is polynomial-time solvable for $H$-free graphs if $H \ssi sP_1+2P_4$ or $H \ssi sP_1+ P_6$ for some $s\geq 0$, and it is \NP-complete otherwise.}

\begin{proof}
If $H$ contains a cycle $C_s$, then we use Lemma~\ref{l-girth2} by setting the girth to $g=s+1$.
Suppose that $H$ contains no cycle, that is, $H$ is a forest.
If $H$ contains a vertex of degree at least~$3$, then we use Lemma~\ref{l-line2}, as in that case the class of $H$-free graphs contains the class of $K_{1,3}$-free graphs, which in turn contains the class of line graphs.
In the remaining cases, $H$ is a linear forest.
If $H\ssi sP_1+P_6$ for some $s\geq 0$ we use Lemma~\ref{l-kp6}.
If $H\ssi sP_1+2P_4$ for some $s\geq 0$ we use Lemma~\ref{l-2p4}.
Otherwise we have that $3P_2\ssi H$ or $P_7\ssi H$ and we use Lemma~\ref{l-3p2p7}. \qed
\end{proof}

\section{A Slight Problem Generalization}\label{s-pro}

In this section we consider a more general variant of the problem. So far, we required that the terminals must all form an independent set.
This condition has been relaxed in some papers in the literature, such as~\cite{LLMT09} (see also Section~\ref{s-intro}).
Given a graph~$G$, we say that \emph{vertex-disjoint} paths $P^1,\ldots, P^k$, for some integer $k\geq 1$, with set $R$ of endpoints are {\it flexibly mutually induced paths} of $G$ if there exists a set $S\subseteq V\setminus R$ such that $G[S \cup R]=(P^1+\ldots +P^k) \cup G[R]$. So, there is no edge between two vertices from different paths $P^i$ and $P^j$ except possibly between the endpoints of the paths.
We can now define the following decision problem:

\problemdef{Flexibly Induced Disjoint Paths}{a graph $G$ and terminal pair collection $T=\{(s_1,t_1)\ldots,(s_k,t_k)\}$.}{does $G$ have a set of flexibly mutually induced paths $P^1,\ldots, P^k$ such that $P^i$ is an $s_i$-$t_i$ path for $i\in \{1,\ldots,k\}$?}

\noindent
Requiring terminals to form an independent set is crucial for our quasipolynomial results. Namely, Theorem~\ref{thm:IDPnew} is unlikely to hold in the relaxed setting, as shown below.

\begin{figure}[tb]
\begin{center}
\includegraphics[scale=0.56]{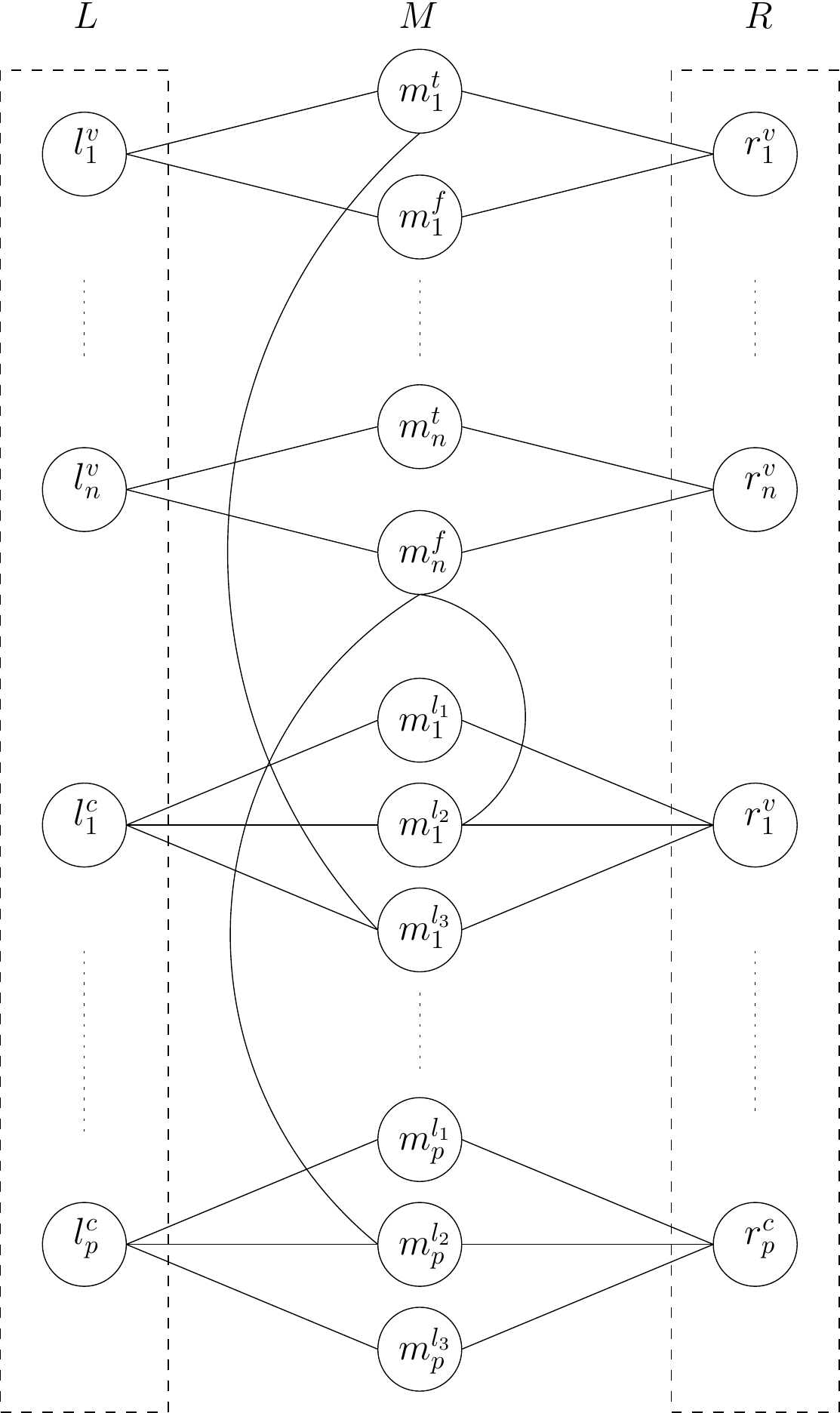}
\end{center}
\caption{The hardness construction for {\sc Flexibly Induced Disjoint Paths} on $P_{14}$-free graphs. The sets $L$ and $R$ are cliques; the corresponding edges are not drawn. The three curved edges between $m$-vertices correspond to an instance of {\sc $3$-Satisfiability} where variable $v_1$ occurs negatively in $c_1$ (as the third literal), variable $v_n$ occurs positively in $c_1$ (as the second literal) and positively in $c_p$ (as the second literal).}
\label{fig:flex-idp-construct}
\end{figure}

\begin{theorem}\label{t-p14}
The {\sc Flexibly Induced Disjoint Paths} problem is \NP-complete for the class of $P_{14}$-free graphs.
\end{theorem}

\begin{proof}
We reduce from {\sc $3$-Satisfiability}, which is well known to be NP-hard. Let $\Phi$ be an instance of {\sc $3$-Satisfiability} with $n$ variables $v_1, \ldots, v_n$ and $p$ clauses $c_1, \ldots, c_p$. We may assume that each variable occurs at most once in each clause. Create a set $L$ of $n+p$ vertices, denoted $l^v_1,\ldots,l^v_n,l^c_1,\ldots,l^c_p$, and a set $R$ of $n+p$ vertices, denoted $r^v_1,\ldots,r^v_n,r^c_1,\ldots,r^c_p$. We make $L$ into a clique and $R$ into a clique. For each variable $v_i$, create two vertices $m^t_i$ and $m^f_i$, which we both make adjacent to $l^v_i$ and $r^v_i$. For each clause $c_j$ and each literal $\ell$ of $c_j$, create a new vertex $m^\ell_j$, which we make adjacent to both $l^c_j$ and $r^c_j$. If $\ell$ is the negation of variable $v_i$, make $m^\ell_j$ adjacent to $m^t_i$; if $\ell$ is variable $v_i$, make $m^\ell_j$ adjacent to $m^f_i$. Call $M$ the set of these $m$-vertices for the variables and the literals. Call the resulting graph $G$. Let the terminal pair collection $T = \{(l^v_1,r^v_1),\ldots,(l^v_n,r^v_n),(l^c_1,r^c_1), \ldots, (l^c_p,r^c_p)\}$. The construction is illustrated in Figure~\ref{fig:flex-idp-construct}. We claim $(G,T)$ is a yes-instance if and only if $\Phi$ is satisfiable.

First suppose that $(G,T)$ is a yes-instance. Let $P^v_1,\ldots,P^v_n,P^c_1,\ldots,P^c_p$ be a solution for the paths between $(l^v_1,r^v_1),\ldots,(l^v_n,r^v_n),(l^c_1,r^c_1), \ldots, (l^c_p,r^c_p)$ respectively. For $1 \leq i \leq n$, since terminal $l^v_i$ is adjacent to other terminals and to $m^t_i$ and $m^f_i$, we know that $P^v_i$ contains one of $m^t_i$ and $m^f_i$ immediately after $l^v_i$. Since $m^t_i$ and $m^f_i$ are adjacent to $r^v_i$, we may assume without loss of generality that $P^v_i$ then continues directly to $r^v_i$. We create a truth assignment $\sigma$ where we set $v_i$ to true if and only if $P^v_i$ contains $m^t_i$. Similarly, we can argue that $P^c_j$ goes from $l^c_j$ to a vertex $m^\ell_j$ for some literal $\ell$ in $c_j$, and then continues to $r^c_j$. If $\ell$ is the negation of $v_i$, then $m^\ell_j$ is adjacent to $m^t_i$. Hence, $m^t_i$ is not in $P^v_i$ and thus the clause is satisfied by $\sigma$. Otherwise, if $\ell$ is $v_i$, then $m^\ell_j$ is adjacent to $m^f_i$. Hence, $m^f_i$ is not in $P^v_i$ and thus the clause is satisfied by $\sigma$. It follows that each clause is satisfied by $\sigma$ and thus $\Phi$ is satisfiable.

Now suppose that $\Phi$ is satisfiable. Let $\sigma$ be a truth assignment that satisfies every clause of $\Phi$. For each variable $i$, let $P^v_i$ be the path from $l^v_i$ to $r^v_i$ that goes via $m^t_i$ if $\sigma(i)$ is set to true and goes via $m^f_i$ otherwise. For each clause $1 \leq j \leq p$, let $P^c_j$ be the path from $l^c_i$ to $r^c_i$ that goes via $m^\ell_j$, where $\ell$ is any literal in $c_j$ that is satisfied by $\sigma$. Observe that when $\ell$ is satisfied, then if $\ell$ is the negation of $v_i$, then $m^\ell_j$ is adjacent to $m^t_i$ but $m^t_i$ is not in $P^v_i$. Similarly, if $\ell$ is $v_i$, then $m^\ell_j$ is adjacent to $m^f_i$ but $m^f_i$ is not in $P^v_i$. It follows that $P^v_1,\ldots,P^v_n,P^c_1,\ldots,P^c_p$ is a set of flexibly mutually induced paths. Hence, $(G,T)$ is a yes-instance.

It remains to argue that $G$ is $P_{14}$-free. Consider a longest induced path $P$ in $G$. Since both $L$ and $R$ are cliques, $P$ contains at most two vertices of $L$ and at most two vertices of $R$, and if $P$ contains two vertices of $L$ (or $R$), then these must appear consecutively. We also note that the vertices in $M$ corresponding to literals have degree~$3$, and thus when $P$ contains such a vertex $m'$, the next or previous vertex on $P$ must be in $L$ or $R$, or $m'$ is an endpoint of $P$. The vertices in $M$ corresponding to variables can have large degree; however, when $P$ contains such a vertex $m''$, the next or previous vertex on $P$ must be in $L$ or $R$ or must be a vertex in $M$ corresponding to a literal, or $P$ has length~$0$. Hence, at most three vertices in $M$ can lie consecutively on $P$ before (or after) a vertex of $L$ or $R$ must appear or an endpoint of $P$ is reached: the $m$-vertex for a literal, a variable, and a literal consecutively. Therefore, in the worst case, $P$ contains three vertices of $M$, followed by two of $L$ or $R$, followed by three of $M$, followed by two of $L$ or $R$, followed by three of $M$. Hence, $P$ has at most~$13$ vertices and thus, $G$ is $P_{14}$-free. \qed
\end{proof}

\section{Future Work}\label{s-con}

We proved a number of new complexity results on induced paths and subgraphs connecting terminals. These results
naturally lead to some open problems. First of all, can we find polynomial-time algorithms for the quasipolynomial cases in Theorem~\ref{thm:IDPnew}? This is a challenging task that is also open for {\sc Independent Set}; note that we reduce to the latter problem in our proof for the case where $H=sP_1+P_6$ for some $s\geq 0$. Interesting open cases are when $H\in \{2P_4,P_7\}$.

We also recall that  the case $H=P_6$ is essentially the only remaining open case left in Theorem~\ref{thm:IDCS}, which is for the setting where $k$ and $\ell$ are both part of the input. As shown in Theorems~\ref{thm:IDPnew} and~\ref{thm:k-IDCS}, respectively, we have a positive answer for the settings where $\ell$ is fixed (and $k$ is part of the input) and where $k$ is fixed (and $\ell$ is part of the input), respectively. However, it seems challenging to combine the techniques used for proving these results for $H=P_6$ when both $k$ and $\ell$ are part of the input.

We did not yet discuss the {\sc $k$-Induced Disjoint Connected $\ell$-Subgraphs} problem, which is the variant where both $k$ and $\ell$ are fixed; note that if $\ell=2$, then we obtain the {\sc $k$-Induced Disjoint Paths} problem. The latter problem restricted to $k=2$ is closely related to the problem of deciding if a graph contains a cycle passing through two specified vertices and has been studied for hereditary graph classes as well; see~\cite{LLMT09}.
Recently, we made some more progress on {\sc $k$-Induced Disjoint Paths}, as we discuss below.

A {\it subdivided claw} is obtained from a claw after subdividing each edge zero or more times. In particular, the {\it chair} is the graph obtained from the claw by subdividing one of its edges exactly once. The set ${\cal S}$ consists of all graphs with the property that each of their connected components is either a path or a subdivided claw. We proved in~\cite{MPSV22b} that for every integer $k\geq 2$ and graph $H$, {\sc $k$-Induced Disjoint Paths} is polynomial-time solvable if $H$ is a subgraph of the disjoint union of a linear forest and a chair, and it is \NP-complete if $H$ is not in ${\cal S}$. 

From the above it follows in particular that $k$-{\sc Induced Disjoint Paths} is polynomial-time solvable for claw-free graphs (just like {\sc Independent Set}~\cite{Mi80,Sh80}). This is in contrast to the three problems in this paper, which are \NP-complete for claw-free graphs (see Theorems~\ref{thm:IDPnew}--\ref{thm:k-IDCS}).
We leave completing the classification of {\sc $k$-Induced Disjoint Paths} as future work and refer to~\cite{MPSV22b} for a more in-depth discussion. 

We also leave classifying {\sc Flexibly Induced Disjoint Paths} to future~research; recall that Theorem~\ref{thm:IDPnew} is unlikely to hold for this problem.

\medskip
\noindent
{\it Acknowledgments.} We thank Pawe{\l} Rz{\k a}\.{z}ewski for the argument using blob graphs, which simplified two of our proofs and led to the case $H=P_6$ in Theorem~\ref{thm:IDPnew}.

\end{document}